\documentclass[twoside]{amsart}
\usepackage{amsfonts,latexsym}
\usepackage{amsmath}
\usepackage{amscd}
\usepackage{float,amsmath,amssymb,mathrsfs,bm,multirow,graphics}
\usepackage[dvips]{graphicx}
\usepackage{amssymb}
\usepackage{amscd}
\usepackage[utf8]{inputenc}
\usepackage[english]{babel}
 
\usepackage{amsthm}

\numberwithin{equation}{section}

\def\b{\begin{eqnarray}}
\def\e{\end{eqnarray}}

\newcommand{\R}{{\mathbb R}}

\newcommand{\N}{{\mathbb N}}
\newcommand{\C}{{\mathbb C}}

\newcommand{\Z}{{\mathbb Z}}

\newcommand{\HH}{{\mathcal H}}
\newcommand{\FF}{{\mathcal F}}

\theoremstyle{plain}
\newtheorem{theorem}{Theorem}[section]
\newtheorem{corollary}[theorem]{Corollary}
\newtheorem{lemma}[theorem]{Lemma}

\theoremstyle{definition}

\newtheorem{example}[theorem]{Example}

\theoremstyle{remark}

\renewcommand{\Re}{{\rm Re}\,}
\renewcommand{\Im}{{\rm Im}\,}
\renewcommand{\phi}{\varphi}
\renewcommand{\epsilon}{\varepsilon}
\newcommand{\vep}{\varepsilon}

\begin{document}
\sloppy
\title[De Branges spaces and Fock spaces]
{De Branges spaces and Fock spaces}

\author{Anton Baranov, H\'el\`ene Bommier-Hato}
\address{Anton Baranov,
\newline Department of Mathematics and Mechanics, St.~Petersburg State University, St.~Petersburg, Russia,
\newline National Research University Higher School of Economics, St.~Petersburg, Russia,
\newline {\tt anton.d.baranov@gmail.com}
\smallskip
\newline \phantom{x}\,\, H\'el\`ene Bommier-Hato,
\newline I2M, Aix-Marseille Universit\'e, CNRS, Marseille, France,
\newline Faculty of Mathematics, University of Vienna, Vienna, Austria,
\newline {\tt helene.bommier@gmail.com}
}
\thanks{The work is supported by Russian Science Foundation grant 14-41-00010.}

\begin{abstract}
Relations between two classes of Hilbert spaces of entire functions, 
de Branges spaces and Fock-type spaces with non-radial weights, are studied. 
It is shown that any de Branges space can be realized as a Fock-type space 
with equivalent area norm, and several constructions of a representing weight 
are suggested. For some special classes of weights (e.g., weights depending
on the imaginary part only) the corresponding de Branges spaces are explicitly 
described.
\end{abstract}

\maketitle

\section{Introduction}

Fock-type spaces and de Branges spaces are arguably two main examples
of reproducing kernel Hilbert spaces (RKHS) of entire functions. 
They include as special cases the Bargmann--Segal--Fock space
(famous for its connections with quantum mechanics and time-frequency analysis) 
and the Paley--Wiener space of bandlimited functions. The aim of the present paper 
is to establish a link between these two classes of spaces.

\subsection{Fock and de Branges spaces}
Let us give the necessary definitions. 
For a measurable function $W: \C \rightarrow(0,\infty)$, 
one defines the {\it Fock-type space}
$$
      \FF_W:=\left\{F\text{ entire}: \ 
      \left\|F\right\|^{2}_{\FF_W}=\int_{\C}\left|F(z)W(z)\right|^2dm(z)<\infty\right\},
$$
where $dm$ stands for the area Lebesgue measure in the complex 
plane $\mathbb{C}$. 
A special attention in the literature was addressed to radial weights (i.e., 
$W(z) = W(|z|)$; see, e.g., \cite{S, bdk, bom} and references therein. 
The classical Bargmann--Segal--Fock space corresponds to $W(z)= e^{-\frac{|z|^2}{2}}$.
However, in the present paper we will be mainly interested in non-radial weights.

An entire function $E$ is said to be in the {\it Hermite--Biehler class} 
if it satisfies $\left|E(z)\right|>\left|E^{\sharp}(z)\right|$ for $z\in\C^+$. 
Here $\C^+$ denotes the upper half-plane and $E^\sharp(z) = \overline{E(\overline{z})}$.
We also always assume in what follows that $E$ does not vanish on $\mathbb{R}$.

Any such function $E$ determines a {\it de Branges space} (see \cite{db})
$$
	\HH(E)=\left\{F\text{ entire}:\ \frac{F}{E},\frac{F^{\sharp}}{E}\in 
        H^{2}(\C^+)\right\},
$$
($H^2= H^2(\C^+)$ being the standard Hardy space in the upper half-plane)
which is a Hilbert space when equipped with the norm 
$$
\left\|F\right\|^{2}_{E}=\int_{\R} \left|\frac{F(x)}{E(x)}\right|^{2}dx.
$$
There exist alternative descriptions of de Branges spaces:
\medskip
\begin{itemize} 
\item 
Axiomatic (see \cite[Theorem 23]{db}) -- any reproducing kernel Hilbert space
of entire functions $\mathcal{H}$ such that the mapping $F\mapsto
F^\sharp$ preserves the norm in $\mathcal{H}$ and the mapping $F\mapsto
\frac{z-\overline w}{z-w}F(z)$ is an isometry in $\mathcal{H}$
whenever $w\in\C\setminus\mathbb{R}$, $F(w) = 0$, is of the form
$\HH(E)$ for some $E$ in the Hermite--Biehler class); 
\smallskip
\item
Via the {spectral data} (see Section \ref{spect}).
\end{itemize}
\medskip

Spaces $\HH(E)$ were introduced by L.~de~Branges in his celebrated solution of 
the inverse spectral problem for canonical systems of differential equations.
At the same time, these spaces are 
a very interesting object from the function theory point 
of view (see, e.g., \cite{ls, os, hm1, hm2, kw, bbb}). The Paley--Wiener space $PW_a$ 
(which consists of entire functions of exponential type 
at most $a$ which are square integrable on $\mathbb{R}$) is a de Branges space corresponding to
$E(z) = e^{-iaz}$.

Any de Branges space is a RKHS (i.e., 
evaluation functionals are bounded). The Fock space $\FF_W$
also is a RKHS if we assume, e.g.,
that
\begin{equation}\label{bab1}
\inf_{z\in K} |W(z)| >0
\end{equation}
for any compact $K\subset \C$. In what follows we always impose this
restriction on $W$.

Note that if two RKHS coincide as sets then, by the Closed Graph Theorem,
their norms are automatically equivalent.

Recently, in \cite[Theorem 1.2]{bbb} a complete description was found of those 
de Branges spaces which coincide with {\it radial} Fock spaces as sets, 
{\it with equivalence of norms}. (It is clear that a Fock space can not 
isometrically coincide with a de Branges space since division by the Blaschke factor 
$\frac{z-w}{z-\overline w}$ is not an isometry.) These turn out to be very small 
spaces generated by functions
$E$ of slow growth (in particular, 
satisfying $\log|E(z)| =O(\log^2 |z|)$, $|z|\to \infty$). Interestingly, 
this is exactly the class of de Branges spaces where any complete and minimal system of 
reproducing kernels (RK) admits the so-called spectral synthesis (see \cite{bbb}).

Also, in \cite{bbb1} the inverse problem was considered: given a radial Fock space
$\FF_W$ with $W(z) = \exp(-\phi(|z|))$ for some increasing $\phi$, when does it coincide 
with some de Branges space? 
Under some regularity conditions on $\phi$, 
the same growth gives a sharp threshold here: if $\phi(r) = O(\log^2 r)$, then 
$\FF_W$ has a Riesz basis of (normalized) RK
corresponding to real points and, thus, 
coincides with a de Branges space (this is a characteristic property of 
de Branges spaces that they have orthogonal bases of RK).  
If, on the contrary, $\log^2 r = o(\phi(r))$, $r\to\infty$, 
then $\FF_W$ has no Riesz bases of RK and so it does not coincide with 
any de Branges space.

In view of these results one may ask whether it is possible to represent any de 
Branges space as a Fock-type space $\FF_W$ with a {\it non-radial} weight, i.e., to replace 
the norm given by an integral over $\mathbb{R}$ by an equivalent 
area integral. One of the main results of the present paper says 
that it is always possible. However, such weight $W$ 
(we call it a {\it representing Fock weight for $\HH(E)$})
is by no means unique, 
and there arises a question how to choose it in a ``canonical'' way so 
that it encompasses in the most economical way the
properties of the space. This is the question we address in the present paper. 
\bigskip


\section{Main results} 

\subsection{Natural weight and the Paley--Wiener spaces}
We start with the simple special case of the problem:
{\it find a weight $W$ such that $PW_a = \FF_W$}. Recall that all equalities 
of spaces are understood as equalities of sets with equivalence 
(but not equality) of norms.

In what follows we write $U(x)\lesssim V(x)$ (or, equivalently,
$V(x)\gtrsim U(x)$) if there is a constant $C$ such that
$U(x)\leq CV(x)$ holds for all values of the parameter $x$. 
We write $U(x)\asymp V(x)$ if both $U(x)\lesssim V(x)$ and $V(x)\lesssim U(x)$.

Note that for an arbitrary de Branges space $\HH(E)$ we have $F/E\in H^2$ for any 
$F\in \HH(E)$ and so
$$
\sup_{y>0} \int_\R \bigg|\frac{F(x+iy)}{E(x+iy)}\bigg|^2dx = \|F\|^2_E.
$$
Hence, 
$$
   \int_{\C^+} \frac{1}{(y+1)^2} \bigg|\frac{F(x+iy)}{E(x+iy)}\bigg|^2 dxdy \le \|F\|^2_E.
$$
This leads to the following ``natural'' candidate for the weight $W$:
\begin{equation}
   \label{wb}
   \begin{aligned}
   W_0(z) & :=\frac{1}{\left|E(z)\right|\left(1+\Im z\right)}, \qquad & z\in \C^+\cup\R, \\
   W_0(\overline z) & :=  W(z),  & z\in \C^-.
   \end{aligned}
\end{equation}
Since $F^\sharp/E$ also is in $H^2$ we always have $\HH(E) \subset \FF_{W_0}$.

The weight $W_0$ depends on $E$, but usually it will be clear from the 
context for which space it is constructed and so we do not introduce $E$
into the notation. 
Clearly, $(1+y)^{-1}$ in the definition of the weight can be replaced 
by $(1+y)^{-\alpha}$
with any $\alpha\in (1/2, \infty)$, as well as by an arbitrary positive 
function from $L^2(0, \infty)$ with at most power decay. The choice of this function is of 
no importance. 
                
If $E(z) = e^{-iaz}$, then $W_0(z) = (1+|\Im z|)^{-1}e^{-a|\Im z|}$. It turns out that
this weight will define an equivalent norm on $PW_a$. Indeed, let $F\in \FF_W$
and $z = x+iy \in\C$. We denote by $D(z, r)$ the open disc with the center $z$ 
of radius $r$. Applying the subharmonicity property to $|F|^2$ in the disc $D(z, |y|+1)$,
we get
$$
|F(z)|^2 \le \frac{1}{\pi(|y|+1)^2} \int_{D(z, |y|+1)} |F(\zeta)|^2 dm(\zeta)
\lesssim 
e^{2a |y|} \int_{D(z, |y|+1)} |F(\zeta)W(\zeta)|^2 dm(\zeta).
$$                                                      
with a constant independent on $F$ and $z$. 
Hence, $|F(x)|^2 \lesssim \int_{D(x,1)} |F(\zeta)W(\zeta)|^2 dm(\zeta)$
and integrating over $x$ we see that $F\in L^2(\R)$. Also, by the above estimates, 
$F$ is of exponential type at most $a$ and so $F\in PW_a$. Thus, $\FF_{W_0} = PW_a$.

It is interesting to note that for any weight of the form
$W(z) = (1+|\Im z|)^{\alpha} e^{-a|\Im z|}$, $\alpha>0$, the corresponding
Fock space does not coincide with a de Branges space. 
The reason for that is that the space $\FF_W$ has no Riesz bases 
of normalized RK. This result was communicated to us by A.~Borichev. 
It follows also from our Theorem \ref{depy} which says that the Paley--Wiener spaces
are the only de Branges space which have a representing weight depending on $\Im z$. 

As we will see, the above simple argument can not be extended to general 
de Branges spaces: it may happen that $\HH(E) \ne \FF_{W_0}$. 
However, we are able to describe the class of de Branges spaces
for which the ``natural'' weight $W_0$ generates the same space. 
Recall that for any Hermite--Biehler function $E$ one can find a 
decreasing branch of the argument on $\R$ (see \cite{db}), that is there 
exists an increasing continuous function $\varphi$  (so-called {\it phase function})
such that
\begin{equation}
    \label{phase}
    E(x)=\left|E(x)\right|e^{-i\varphi(x)},\qquad x\in \R.
\end{equation}
If $Z_E=\left\{\overline z_n=x_n-iy_n\right\}\subset \C^-$ 
denotes the zero set of $E$ (note that in our notation $z_n \in \C^+$), 
then there exists $a \geq 0$ such that
\begin{equation}\label{def phi'}
  \varphi'(x)=a+\sum_{n}\Im \frac{1}{x- z_n} 
  = a+\sum_{n} \frac{y_n}{|x-\overline z_n|^2}, 
\qquad x\in \R. 
\end{equation}

Now we can state our first main result. 

\begin{theorem}
\label{F Wb} 
Let $\HH(E)$ be a de Branges space and let $W_0$ be defined by \eqref{wb}. Then
the following are equivalent\textup:
\smallskip
\begin{enumerate}
	\item [(i)] The spaces $\HH(E)$  and $\FF_{W_0}$  coincide as sets\textup;
\smallskip
	\item [(ii)] $\varphi' \in L^{\infty}(\R)$.
\end{enumerate}
\end{theorem}

De Branges spaces with the property $\varphi' \in L^{\infty}(\R)$ 
(sublinear growth of the argument) have many
nice properties similar to the classical Paley--Wiener spaces (see \cite{dy} or 
\cite{hm2}).


\subsection{General de Branges spaces}

We have seen that the weight $W_0$ solves the problem only for a special class
of de Branges spaces. In general we need to modify it to encompass the 
subtle behavior of the elements of $\HH(E)$ near $\R$ in the case when zeros of $E$
approach the real axis. We will do this using the level sets of inner
functions.

Note that any Hermite--Biehler function $E$
gives rise to a meromorphic inner function in $\C^+$, namely
\begin{equation}
   \label{def Theta}
   \Theta=\frac{E^{\sharp}}{E},
\end{equation}
which has the form 
\begin{equation}
   \label{merom}
\Theta(z)=e^{i\tau z}\prod_{n}\epsilon_n \frac{z-z_n}{z-\overline z_n},\qquad z\in\C, 
\end{equation}
where $\overline{z}_n$ are zeros of $E$, $\epsilon_n=\frac{\left|z_n^2+1\right|}{z_n^2+1}$ 
(for $z_n=i$, $\epsilon_n=1$) and $\tau$ is a nonnegative constant. Note that 
$|\Theta'(x)| = 2\phi'(x)$, $x\in \R$.

With each inner function $\Theta$ one can associate the {\it model} ({\it backward shift 
invariant}) subspace $K_\Theta = H^2\ominus \Theta H^2$ of the Hardy space $H^2$.
The map $F\mapsto F/E$ is a canonical unitary map from $\HH(E)$ onto $K_\Theta$
where $\Theta$ is defined by \eqref{def Theta} (this is easy to verify; see, 
e.g., \cite[Theorem 2.10]{hm1}). 

The reproducing kernel of $\HH(E)$ is given by 
$$
   K_w(z)=\frac{i}{2\pi}\frac{E(z)\overline{E(w)}-E^{\sharp}(z)
   \overline{E^{\sharp}(w)}}{z-\overline w},\qquad z,w\in \C;
$$
in particular, we have $\|K_x\|^2_E = K_x(x) =
|E(x)|^2 \varphi'(x)/\pi$, $x\in \R$.

Note also that $K_w(z) = E(z)\overline{E(w)} k_w(z)$ where
$$
k_w(z) = \frac{i}{2\pi} \frac{1-\overline{\Theta(w)}\Theta(z)}{z-\overline{w}}
$$
is the reproducing kernel of $K_\Theta$ 
and 
$$
\|k_z\|_2^2 = \frac{1-|\Theta(z)|^2}{4\pi \Im z}, \quad z\in \C^+,\qquad
\|k_x\|_2^2 = |\Theta'(x)|/2\pi, \quad x\in \R;
$$
here $\|f\|_2$ denotes the standard $L^2(\R)$-norm.

Though the function $\Theta$ associated to a de Branges space is 
meromorphic in $\C$, we prefer now to work in a slightly more general setting. 
Let $\Theta$ be an arbitrary inner function in $\C^+$ (see, e.g., \cite{ga}). 
Let $\sigma(\Theta)$ be the so-called {\it spectrum} of $\Theta$, 
that is, the set of all $\zeta\in\overline {\mathbb{C^+}}$
such that $\liminf\limits_{z\to\zeta,\, z\in\mathbb{C^+}}|\Theta(z)|=0$.
Equivalently, $\sigma(\Theta)$ is the smallest closed subset of
${\mathbb{C^+}}\cup \R$ containing the zeros $z_n$ of $\Theta$
and the support of the singular measure associated with the 
singular factor in $\Theta$. In the case of meromorphic $\Theta$,
$\sigma(\Theta)$ is just the zero set $\{z_n\}$.

Given $\Theta$ and $\varepsilon\in (0,1)$, consider the sublevel sets
$$
\Omega_\varepsilon=\{z\in\mathbb{C^+}: |\Theta(z)|<\varepsilon\}
\qquad\text{and} \qquad \Omega^c_\varepsilon = (\C^+\cup\R) \setminus\Omega_\varepsilon.
$$
For $z\in\C^+ \cup \mathbb{R}$ put
$$
d_\varepsilon(z)={\rm dist} (z,\Omega_\varepsilon);\qquad\quad
d_0(z)={\rm dist} (z,\sigma(\Theta)).
$$
Then, for $0<\varepsilon<\delta<1$ there exist constants $A_1,A_2>0$ depending only on 
$\varepsilon, \delta$ (but not on $\Theta$!) such that
\begin{equation}
   \label{lev}
  A_1\min (d_0(z), \|k_z\|_2^{-2}) \le d_\varepsilon(z) \le
  A_2 \min (d_0(z), \|k_z\|_2^{-2}), \qquad z\in \Omega_\delta^c.
\end{equation}
This inequality was established in \cite[Theorem 4.9]{bar1} for the case
$z=x\in\R$. The proof of \eqref{lev} is very similar to the proof in \cite{bar1};
we give its sketch in Section \ref{level}.

Any de Branges space can be realized as a Fock-type space with the 
representing weight depending on the geometry of level sets. By $\chi_e$ 
we denote the usual characteristic function of a set $e$.

\begin{theorem}
\label{main}
Let $\HH(E)$ be a de Branges space, let $W_0$ be defined by \eqref{wb},
and let $\Theta$ be the associated inner function.
For $0<\varepsilon<\delta<1$ put 
\begin{equation}
   \label{mw}
W(z) = W_0(z)\big(1+ (d_\vep(z))^{-1/2} \chi_{\Omega_\delta^c} (z) \big), \qquad
z\in \C^+ \cup\R,
\end{equation}
and $W(z) = W(\overline z)$, $z\in \C^-$. Then $\HH(E) = \FF_W$.
\end{theorem}

An important subclass of inner functions is formed by $\Theta$ such that the set
$\Omega_\vep$ is connected for some $\vep\in(0,1)$. They are usually 
called {\it one-component} inner functions. This class was introduced by 
W.~Cohn \cite{cohn} and studied extensively by A.~Aleksandrov (see, e.g., 
\cite{al1} and references therein). Some results about de Branges spaces whose $\Theta$ 
is one-component can be found in
\cite{MNO}. In particular this paper contains an observation
(communicated to authors of \cite{MNO} by the first author) 
that a meromorphic $\Theta$ is one-component if and only if 
the measure $\phi'(x) dx$ on $\R$ is locally doubling. 

It follows from the results of \cite{bar1} (and will be explained below) that for 
one-component inner function one has $d_\vep(z) \asymp \|k_z\|_2^{-2}$, 
$z\in \Omega_\delta^c$. Thus, we have the following corollary:

\begin{corollary}
\label{one1}
If $\Theta = E^\sharp/E$ is a one-component inner function, then 
for $0<\varepsilon<\delta<1$, we have $\HH(E)=\FF_W$
where $W(z) = W_0(z)\big(1+ \|k_z\|_2\, \chi_{\Omega_\delta^c} (z) \big)$, 
$z\in \C^+ \cup\R$, extended symmetrically to $\C^-$.
\end{corollary}


\subsection{Weights depending on imaginary part}

We have seen that for the Paley--Wiener space $PW_a$,
one can choose a representing Fock 
weight $W_0(z) = (|\Im z|+1)^{-1}e^{-a|\Im z|}$ which depends only on $\Im z$. 
Are there other de Branges spaces with the same property? The answer turns out 
to be negative:

\begin{theorem}
\label{depy}
Let $\HH(E)$ be a de Branges space. Then the following \smallskip  are 
equivalent\textup:
\begin{enumerate}
	\item [(i)] There exists a representing weight $W$ for $\HH(E)$ 
such that $W(z) = W(\Im z)$, $z\in\C$\textup; 
\smallskip
	\item [(ii)] $\HH(E) = PW_a$ for some $a>0$.
\end{enumerate}
\end{theorem}

Note that there exist functions $E$ such that $|E|$ is not equivalent to 1 on $\R$,
but $\HH(E) = PW_a$. We discuss such construction (from paper \cite{ls}) 
in Section \ref{th3}.
\medskip

We conclude with one question which we were not able to answer.
Note that the property (i) implies, in particular, that 
all real translations 
$U_t: F\mapsto F(\cdot+t)$, $t\in\mathbb{R}$, are isomorphisms of $\HH(E)$ onto itself
and $\sup_{t\in\R}\|U_t:\HH(E)\to\HH(E)\|<\infty$. 
It would be interesting to know whether the converse is true, 
that is, whether there exist de Branges spaces where
all real translations are isomorphisms with uniformly bounded norms 
which are different from the Paley--Wiener spaces.

\subsection{Organization of the paper}
In Section \ref{level} we prove some auxiliary estimates for the distances to the 
level sets. 
In Section \ref{th2} we prove the main result of the paper, Theorem \ref{main}.
In Section \ref{th1} we prove necessity of the condition 
$\phi'\in L^\infty$ in Theorem \ref{F Wb} (its sufficiency follows almost 
immediately from Theorem \ref{main}). In Section \ref{spect} we give another construction
of a representing weight for a general de Branges space which uses geometry of the 
spectral data for $\HH(E)$. Finally, in Section \ref{th3} we prove
Theorem \ref{depy} on the weights depending on the imaginary part only. 
\bigskip


\section{Distances to the level sets}
\label{level}

In this section we prove inequality \eqref{lev}. 

\begin{theorem}
\label{lev1}
Let $0<\varepsilon<\delta<1$. Then there exist constants $A_1, A_2>0$
depending on $\epsilon ,\delta$ such that for any inner function $\Theta$ in $\C^+$, 
\begin{equation}
   \label{lev2}
   A_1\min (d_0(z), \|k_z\|_2^{-2}) \le d_\varepsilon(z) \le
   A_2 \min (d_0(z), \|k_z\|_2^{-2}), \qquad z\in \Omega_\delta^c.
\end{equation}
\end{theorem}

\begin{proof}
Without loss of generality, we can assume that $\Theta$ is a Blaschke product. 
Indeed, by the Frostman theorem (see, e.g., \cite[Chapter II]{ga}), Blaschke products are 
uniformly dense in the set of all inner functions. So in what follows $\Theta=B$
is a Blaschke product with zeros $z_n=x_n+i y_n$, $n\in\N$, with $y_n>0$, 
and for each integer $n$, we set 
$$ 
b_n(z)=\frac{z-z_n}{z-\overline z_n},\qquad 
B_n=\prod^{n-1}_{j=1} b_j \qquad (B_0\equiv 1).
$$

Suppose that $z\in \Omega_{\delta}^c$ and 
$\left|w-z\right|<\alpha\min\left(d_{0}(z), \left\|k_z\right\|_2^{-2}\right)$. 
We will show that if $\alpha>0$ is small enough, then $\left|B(w)\right|>\epsilon$. 
This will prove the estimate from below in \eqref{lev2}.

In what follows the constants in estimates $\lesssim$ and $\asymp$
depend only on $\delta$ and $\vep$, but do not depend on $B$.
\medskip
\\
{\bf Estimate of} $d_\vep$ {\bf from below.}
Since $z\in \Omega_{\delta}^c$, we have                  
\begin{equation}\label{esti bn}
	\delta\leq \left|B(z)\right|\leq \left|B_{n+1}(z)\right|
        \leq\left|\frac{z-z_n}{z-\overline z_n}\right|,\qquad n\in\N.
\end{equation}
Using now the identities 
\begin{gather}
  \label{ert}
   2\pi\left\|k_z\right\|_2^{2}=\frac{1-\left|B(z)\right|^2}{\Im z}=\sum_{n}\left|B_{n}(z)\right|^2\frac{1-\left|b_n(z)\right|^2}{\Im z},\\
  \label{ert0}
   \frac{1-\left|b_n(z)\right|^2}{\Im z}=\frac{4 y_n}{\left|z-\overline z_n\right|^2},
\end{gather}
we get that
\begin{equation} 
  \label{esti1 k_z}
   \delta^2\sum_{n}\frac{4 y_n}{\left|z-\overline z_n\right|^2}\leq 
   2\pi\left\|k_z\right\|_2^2\leq\sum_{n}\frac{4 y_n}{\left|z-\overline z_n\right|^2}.
   \end{equation}
Recall that 
$$
|B'(x)| = \sum_{n}\frac{2 y_n}{|x-z_n|^2}, \qquad x\in \R
$$
(of course, in general $|B'(x)|$ need not be finite). It follows from
\eqref{esti bn} that 
\begin{equation}
  \label{ert1}
\delta|z-\overline z_n|\le |x-z_n| \le |z-\overline z_n|
\end{equation}
and so, by \eqref{esti1 k_z}, 
$$
   \delta^4\left|B'(x)\right|\leq 2\pi\left\|k_z\right\|_2^2, \qquad 
z=x+iy\in\Omega_{\delta}^c.
$$
Thus, for $z=x+iy\in\Omega_{\delta}^c$, we have 
\begin{equation}
  \label{esti3 k_z}
    \left\|k_z\right\|_2^2\asymp \left|B'(x)\right|\qquad \text{and}\qquad
    d_{0}(z) \asymp d_{0}(x).
\end{equation}

Now, let $w=u+iv$, $v>0$. We have 
$$
  \log\left|\frac{B(w)}{B(z)}\right|^2 =
  \sum_{n}\log\left[\frac{\left(u-x_n\right)^2+\left(v-y_n\right)^2}
   {\left(u-x_n\right)^2+\left(v+y_n\right)^2} \cdot \frac{\left(x-x_n\right)^2+\left(y+y_n\right)^2}{\left(x-x_n\right)^2+\left(y-y_n\right)^2}\right]
  =\sum_{n} \log\left[1-Y_n\right]
$$
where 
$$
Y_n = \frac{4vy_n\left(x-x_n\right)^2-4yy_n\left(u-x_n\right)^2}
  {\left|z-z_n\right|^{2}\left|w-\overline z_n\right|^{2}}-\frac{\left(y-y_n\right)^2
  \left(v+y_n\right)^2 +
\left(y+y_n\right)^2\left(v-y_n\right)^2}{\left|z-z_n\right|^{2}\left|w-\overline z_n\right|^{2}}.
$$
Since $\left|w-z\right|<\alpha d_0(z)$, we observe that
\begin{equation}\label{esti w-zn bar}
\left(1-\alpha\right)\left|z-\overline z_n\right|\leq \left|w-\overline z_n\right|\leq 
\left(1+\alpha\right)\left|z-\overline z_n\right|.
\end{equation}
Using (\ref{esti bn}) and (\ref{esti w-zn bar}), it is easily seen that
\begin{align*}
\sum_n\left|Y_n\right|&\leq C(\delta)\sum_n\frac{\left|w-z\right|y_n}{\left|z-\overline z_n\right|^2}\\
&\leq \tilde C(\delta)\left|w-z\right|\left\|k_z\right\|_2^2\leq \tilde C(\delta)\alpha.
\end{align*}
Hence, if $\alpha$ is sufficiently small, we conclude that
$$ 
\log\left|\frac{B(w)}{B(z)}\right|^2> 2 \log\frac{\epsilon}{\delta},
$$
whence $\left|B(w)\right|>\epsilon.$
Thus, there exists a positive constant $C=C(\delta,\epsilon)$ such that
$$ 
d_{\epsilon}(z)\geq C\min\left(d_{0}(z),\left\|k_z\right\|_2^{-2}\right). 
$$
\smallskip
\\
{\bf Estimate of} $d_\vep$ {\bf from above.}
Let us prove the converse inequality. Let $z=x+iy\in\Omega_{\delta}^c$. Inequality
\eqref{esti bn} implies that $|z-z_n| \ge \delta (2y -|z-z_n|)$ whence 
$y\lesssim d_{0}(z)$.
Also, $y = (2\pi\|k_z\|_2^2)^{-1} (1-|B(z)|^2) \lesssim \|k_z\|_2^{-2}$. Hence,
\begin{equation}
  \label{ert2}
  y\lesssim \min\left(d_{0}(z),\left\|k_z\right\|_2^{-2}\right).
\end{equation}
In view of \eqref{esti3 k_z}, it remains to show that
$$ 
  d_{\epsilon}(x) \lesssim \min\left(d_{0}(x),\left|B'(x)\right|^{-1}\right) 
  \asymp \min\left(d_{0}(z),\left\|k_z\right\|_2^{-2}\right).
$$
The latter estimate is proved in \cite[Theorem 4.9]{bar1}. We 
include its proof to make the exposition self-contained. 

Put  $\beta=(1+4\log\frac{1}{\varepsilon})$.
Clearly, $d_0(x) \ge d_\varepsilon(x)$. Thus,
if $2\beta|B'(x)|^{-1}\ge  d_0(x)$,
then $d_\varepsilon(x) \le 2 \beta \min\left(d_{0}(x), |B'(x)|^{-1}\right)$. 
Assume that $2\beta|B'(x)|^{-1}< d_0(x)$ and put $w=x+i\beta|B'(x)|^{-1}$. 
Then $|x-w|<d_0(x)/2$ and so 
$$
|x-z_n|/2 < |w-\overline z_n|< 2|x-z_n|
$$ 
for any $n$. We have also
$$
\begin{aligned}
\log|B(w)|^2 & = \sum_n \log |b_n(w)|^2 = 
\sum_n \log \bigg(1- \frac{4y_n\Im w}{|w-\overline z_n|^2}\bigg) \\
& < -\sum_n \frac{4\delta|B'(x)|^{-1}y_n}{|w-\overline z_n|^2} \le
-\sum_n \frac{\beta|B'(x)|^{-1}y_n}{|x-\overline z_n|^2}
= -\beta/2<2\log\varepsilon.
\end{aligned}
$$
Hence, $|B(w)|<\varepsilon$, and so 
$d_\varepsilon(x)<|x-w|=\beta |B'(x)|^{-1}$.
\end{proof}

In the case when $\Theta$ is a one-component inner function, we have 
a simpler formula:

\begin{corollary} 
\label{onecom}
Let $\Theta$ be a one-component inner function in $\C^+$ and let $0<\vep<\delta<1$. 
Then 
$$
d_\vep(z) \asymp \|k_z\|_2^{-2} \asymp |\Theta'(\Re z)|^{-1}, \qquad z\in \Omega_\delta^c,
$$
where the constants in the asymptotic equality depend on $\Theta$, $\vep$ and 
$\delta$, but not on $z\in \Omega_\delta^c$.
\end{corollary}

\begin{proof} 
By a result of Aleksandrov \cite[Theorem 1.2]{al1}, $\Theta$ is one-component
if and only if there exists a constant $A>0$
such that $\|k_z\|_\infty \le A \|k_z\|_2^2$, $z\in \C^+$.  
Note that this estimate implies that $k_x \in L^2(\R)$ for $x\in \R$
only if $x$ is separated from the spectrum  of $\Theta$.
Then
$$
|z_n- \overline{z}|^{-1} = 2\pi |k_z(z_n)| \le 2\pi A \|k_z\|_2^2,
$$
whence $|z_n-\overline z| \gtrsim \|k_z\|_2^{-2}$, $z\in \C^+$.
In particular, $|x-z_n| \gtrsim |\Theta'(x)|$, $x\in \R$,
and so $d_0(x) \gtrsim |\Theta'(x)|$. It follows that $d_\vep(x) \asymp |\Theta'(x)|$, 
$x\in \R$.

Now let $z = x+iy \in\Omega_\delta^c$. Then we know from \eqref{esti3 k_z}
that $\|k_z\|_2^2 \asymp |\Theta'(x)|$ and $d_\vep(z) \asymp d_\vep(x)$. 
This completes the proof.
\end{proof}
\bigskip


\section{Proof of Theorem \ref{main}}
\label{th2}

In what follows it will be convenient to use the following 
equivalent description of the Hardy space $H^2$. We will say that $f$ is a 
{\it function of bounded type} in the upper half-plane if $f$ is analytic in $\C^+$
and $f=g/h$ where $g, h \in H^\infty$ (here $H^\infty$ is the space of 
all bounded analytic functions in $\C^+$). If, moreover, $h$ is outer 
(see, e.g., \cite[Chapter II]{ga})  then we say that $f$ belongs to the Smirnov class
$\mathcal{N}_+$. A classical theorem of V.I. Smirnov says that $f\in H^2$ if and only 
if $f\in \mathcal{N}_+$ and $f\in L^2(\R)$ in the sense of nontangential boundary values. 

If $f$ is analytic in $\C^+$ and continuous in $\C^+\cup \R$, then 
$f\in \mathcal{N}_+$
if and only if  $f$ is of bounded type and 
$$
\limsup\limits_{y\to+\infty}\frac{\log|f(iy)|}{y} \le 0
$$
($f$ is of nonpositive {\it mean type}).

In the proof of Theorem \ref{main} we will need the following

\begin{lemma}
\label{phi}
Let $f$ be analytic in $\C^+$ and continuous on $\C^+\cup \R$. 
Assume that there exists a nonincreasing function 
$\psi:(0,+\infty)\rightarrow(0,+\infty)$ such that,
for some $r>0$,
$$ 
\int^{r}_{0}y\left|\log \psi(y)\right|dy<\infty
$$
and that $\left|f(z)\right|\leq\psi(y)$, $z=x+iy \in \C^+$. If, moreover, 
$f\in L^2(\R)$, then $f\in H^2$.
\end{lemma}

This lemma is a consequence of a remarkable result of L.~de~Branges \cite[Theorem 11]{db}
which gives sufficient condition for the inclusion to the Smirnov class $\mathcal{N}_+$.
We will need Lemma \eqref{phi} only for $\psi(y)=1/y$, but prefer 
to give a proof in full generality (note that the critical growth is 
$\psi(y) = \exp(1/y^2)$).

\begin{proof}
By Smirnov's theorem, it is sufficient to show that $f$ is in the Smirnov class 
$\mathcal{N}_+$. By de Branges' theorem  \cite[Theorem 11]{db}, 
a function $f$ (analytic in $\C^+$ and continuous in $\C^+\cup\R$)
is in  $\mathcal N_+$ if the following conditions hold:
\begin{description}
	\item[(a)] $\lim_{y\to \infty}\dfrac{\log\left|f(iy)\right|}{y}\leq 0$;
	\item[(b)] $\int_{\R}\dfrac{\log^+\left|f(t)\right|}{1+t^2}dt<\infty$;
\medskip
	\item[(c)] $\lim_{R\to \infty} \frac{1}{R^2} \int^{\pi}_{0}\log^+\left|f(R e^{i\theta})\right|\sin \theta\, d\theta=0. $
\end{description}             
\smallskip
Thus, apart from obvious necessary conditions (a) and (b) one needs to satisfy 
a very mild growth restriction (c).

In our case conditions (a) and (b) are immediate. Condition (c) follows  
from the inequality $\left|f(z)\right|\leq\psi(y)$, for $z=x+iy$. Indeed,
$$ 
\int^{\pi - \text{arcsin}\frac{1}{R}}_{\text{arcsin}\frac{1}{R}}
\log^+\left|f(R e^{i\theta})\right|d \theta\lesssim R,
$$
since in this integral $\Im(R e^{i\theta})\ge 1$ and $f$ 
is bounded in $\{\Im z\ge 1\}$, while
$$ 
\int^{\text{arcsin}\frac{1}{R}}_{0}\log^+\left|f(R e^{i\theta})\right |
\sin \theta d \theta\lesssim \frac{1}{R}\int^{1}_{0} s \log^+\psi(s)ds
$$
(in the last integral, we have used the change of variable $s=R\sin \theta$). 
Thus $f\in H^2$.
\end{proof}

\begin{proof}[Proof of Theorem \ref{main}]
For the proof, it will be useful to consider a special system of intervals and Carleson squares. We may cover the real line as $\R=\cup_{n}I_n$, where $(I_n)_n$ are intervals with mutually disjoint interiors such that 
$$ 
\left|I_n \right| \asymp\text{dist}\left(I_n, \Omega_{\delta}\right)\quad 
\text{ and } \quad S(I_n)\subset \Omega_{\delta}^c,
$$
where $S(I)=I\times [0,\left|I\right|]$
is the Carleson square. Note that $\left|I_n \right|\asymp\left|I_{n+1} \right|$ 
for two contiguous intervals. We also have 
\begin{gather*}
   d_{\delta}(x)\asymp d_{\delta}(y),\qquad x,\ y\in I_n,\\
   d_{\delta}(x)\asymp d_{\epsilon}(x), \qquad x\in\R,
\end{gather*}
where the involved constants depend on $\epsilon$ and $\delta$ only. Then 
it follows from \eqref{esti3 k_z} that
$$
d_{\epsilon}(z)\asymp\left|I_n \right|, \qquad z\in S(I_n).
$$

We shall consider a modified weight, 
\begin{equation}\label{modified W}
	\widetilde W(z)= W_0(z) \left(1+\sum_{n}\frac{1}{ \left|I_n \right|^{1/2}}\chi_{S(I_n)
	}(z)\right), \qquad z\in \C^+\cup \R,
\end{equation}
and $\widetilde W(z)=\widetilde W(\overline z),\ z\in \C^-$.
It is clear that $\widetilde W \lesssim W$ and so $\FF_W \subset 
\FF_{\widetilde W}$. \\

{\bf Inclusion} $\FF_{\widetilde W}\subset \HH(E)$. 
We need to show that $F/E\in L^2(\R)$, and also that 
$F/E, F^{\sharp}/E$ are in $H^2(\C^+)$. Note that by 
dividing the intervals $I_n$ in smaller parts, we can achieve 
that $S(2I_n)\subset \Omega_{\delta}^c$, where by $\gamma I_n$ we always denote 
the interval of length $\gamma \left|I_n\right|$ with the same center as $I_n$.

Let $\alpha\in\left(0,\frac{1}{2}\right)$. 
For $x\in I_n$, put $D(x)=D(x,\alpha\left|I_n\right|)$ if $\left|I_n\right|\leq 1$, 
and $D(x)=D(x,\alpha)$ otherwise. 

Note that $\left|\Theta(z)\right|\geq\delta$, 
whenever $z\in S(2I_n)$. Since $\Theta(z)=\frac{1}{\overline{\Theta(\overline z)}}$ 
for $z\in\C^{-}$, we have $\left|\Theta(z)\right|\leq\frac{1}{\delta}$ when 
$z\in D(x)\cap \C^{-}$, since $D(x)\subset S(2I_n)\cup \overline{S(2I_n)}$
(where, for a set $A\subset \mathbb{C}$ we put 
$\overline{A} = \{\overline z: z\in A\}$). 
In particular, 
$\left|E(z)\right|\geq\delta\left|E(\overline z)\right|$, for $z\in D(x)\cap \C^{-}$.

We first assume that $\left|I_n\right|\leq 1$. Using subharmonicity, we have 
$$  
\left|\frac{F(x)}{E(x)}\right|^{2}\leq\frac{1}{\pi\alpha^2\left|I_n\right|^2}\int_{D(x)}\left|\frac{F(z)}{E(z)}\right|^{2}dm(z).
$$
Integrating over $I_n$, we get
\begin{align*}
\int_{I_n}\left|\frac{F(x)}{E(x)}\right|^{2}dx&\leq \frac{1}{\pi\alpha^2\left|I_n\right|^2}\int_{I_n}\left(\int_{D(x)}\left|\frac{F(z)}{E(z)}\right|^{2}dm(z)\right)dx\\
&\lesssim \frac{1}{\left|I_n\right|}\int_{\left(\alpha+1\right)I_n\times [-\alpha\left|I_n\right|, \alpha\left|I_n\right|]}\left|\frac{F(z)}{E(z)}\right|^{2}dm(z)\\
&\lesssim \int_{\left(\alpha+1\right)I_n\times [-\alpha\left|I_n\right|, \alpha\left|I_n\right|]}\left|F(z)\widetilde W(z)\right|^2dm(z),
\end{align*}
since $\left|E(z)\right|\geq\delta\left|E(\overline z)\right|$ 
for $z\in \C^{-}$ and $\widetilde W(z)\asymp \frac{1}{\left|E(z)\right|}
\frac{1}{\left|I_n\right|^{1/2}}$ in $S(2I_n) \cup \overline{S(2I_n)}$.

In the case when $\left|I_n\right|>1$, we have
$$  
\left| \frac{F(x)}{E(x)}\right|^{2} \leq\frac{1}{\pi\alpha^2}
\int_{D(x)}\left|\frac{F(z)}{E(z)}\right|^{2}dm(z)
$$
whence, integrating over $x\in I_n$, we get
\begin{align*}
\int_{I_n}\left|\frac{F(x)}{E(x)}\right|^{2}dx \lesssim
\int_{\left(\alpha+1\right)I_n\times [-\alpha, 
\alpha]}\left|F(z)\widetilde W(z)\right|^2dm(z),
\end{align*}
since now $\widetilde W\asymp \frac{1}{|E|}$ in 
$(\alpha+1)I_n\times [-\alpha, \alpha]$.

Recall that $|I_n| \asymp |I_{n+1}|$ for two contiguous intervals. 
Therefore the sets $\left(\alpha+1\right)I_n\times [-\alpha |I_n|, 
\alpha |I_n|]$ for $|I_n|\le 1$ (or $(\alpha+1)I_n\times [-\alpha, 
\alpha]$ for $|I_n|>1$) 
may intersect but with finite multiplicity, which implies that
$$  
\sum_{n}\int_{I_n}\left|\frac{F(x)}{E(x)}\right|^{2}dx\lesssim \int_{\C}\left|F\widetilde W\right|^2 dm,$$
and so
$$ 
\int_{\R}\left|\frac{F(x)}{E(x)}\right|^{2}dx\lesssim \int_{\C}\left|F\widetilde W\right|^2 dm.
$$

To show that $F/E$ is in $H^2(\C^+)$, we can use Lemma \ref{phi}
and the fact that $\widetilde W \ge W_0$.  
For $z\in\C^+$, we obtain
\begin{align*}
\left|\frac{F(z)}{E(z)}\right|^{2}&\leq\frac{1}{\pi\left(\Im z\right)^2}
\int_{D(z,\Im z/2)}\left|\frac{F}{E}\left(\zeta\right)\right|^2dm\left(\zeta\right)\\
&\lesssim\frac{(\Im z+1)^2}{\left(\Im z\right)^2}\int_{D(z,\Im z/2)}\left|FW_0\right|^2dm
\le \frac{(\Im z+1)^2}{\left(\Im z\right)^2}\left\|F\right\|^{2}_{\FF_{\widetilde W}}.
\end{align*}
By Lemma \ref{phi} applied to $\phi(y) = (y+1)/y$, 
we have $F/E\in H^2$; the reasoning for  $F^{\sharp}$ is analogous.\\

{\bf Inclusion} $\HH(E)\subset \FF_W$. We have already seen that if $F/E\in H^2$, then
$$ 
\int_{\C^+}\left|\frac{F(z)}{E(z)} \right|^{2}\frac{1}{\left(1+\Im z\right)^2}dm(z)<\infty.
$$
We need to show that
$$
\int_{ \Omega_{\delta}^c }\left|\frac{F(z)}{E(z)}\right|^{2}
\frac{1}{d_{\epsilon}(z)}dm(z)\lesssim \left\|F\right\|^{2}_{\HH(E)},
$$
i.e., that 
$$
d\mu(z):=\frac{1}{d_{\epsilon}( z)}\chi_{\Omega_{\delta}^c}(z)dm(z)
$$
is a Carleson measure for the model space $K_{\Theta}$.

As before, consider the system of intervals $I_n$ such that $S(2I_n)\subset \Omega_{\delta}^c$ and 
$$\left|I_n\right|\asymp  \text{dist}\left(I_n, \Omega_{\delta}\right)\asymp  \text{dist}\left(I_n, \Omega_{\epsilon}\right).$$
In order to show that  $\mu$ is a Carleson measure for $K_{\Theta}$, 
we will verify the conditions of the following theorem of 
A.L.~Volberg and S.R.~Treil \cite{vt}: {\it if for any $I$ 
with $S(I)\cap \Omega_{\epsilon}\neq \varnothing$, 
we have $\mu\left(S(I)\right)\lesssim \left|I\right|$ 
\textup(with constants independent  of $I$\textup), then $\mu$ is 
a Carleson measure for $K_{\Theta}$.}

Let $I = [a,b]$ be an interval such that  
$S(I)\cap \Omega_{\epsilon}\neq \varnothing$.  
If $a\in I_n$, then $\left|I\right|\geq c \left|I_n\right|$. 
We can write $I\subset \cup_{n\in\mathcal J}I_n$, for some finite set  $\mathcal J$ and $\left|I\right|\geq\sum_{n\in\mathcal J}\left|I_n\right|$.

If $z=x+iy\in \Omega_{\delta}^c$ with $x\in I_n$, we know that 
$d_{\epsilon}(z)\asymp d_{\epsilon}(x) \asymp \left|I_n\right|$, 
with constants independent on $I_n$ and $z$. 
We also have $y\lesssim d_{\epsilon}(z) \lesssim \left|I_n\right|$
by \eqref{ert2}. Hence,
$$
\mu\big((I_n\times [0,+\infty))\cap \Omega_{\delta}^c \big)=
\int_{I_n\times [0,+\infty)\cap \Omega_{\delta}^c}\frac{1}{d_{\epsilon}(z)}dm(z)
\lesssim \left|I_n\right|.
$$
Therefore,
$$ 
\mu\left(S(I)\right) \leq\sum_{n\in\mathcal J}\mu\big((I_n\times [0,+\infty))
\cap \Omega_{\delta}^c\big)\lesssim \sum_{n\in\mathcal J}\left|I_n\right|\lesssim \left|I\right|.
$$
Thus $\mu$ is a Carleson measure for $K_\Theta$.
\end{proof}

Corollary \ref{one1} for one-component inner function follows immediately from 
Theorem \ref{main} and Corollary \ref{onecom} since 
in this case $d_\vep(z) \asymp \|k_z\|_2^{-2}$, $z\in\Omega_\delta^c$.
Using that $\|k_z\|_2^{2} \asymp|\Theta'(x)| =2\phi'(x)$, $z=x+iy\in 
\Omega_\delta^c$, one can easily deduce another representation for 
the one-component case:

\begin{corollary}
\label{one2}
Let $\Theta = E^\sharp/E$ be one-component. Put 
$$
W_1(z) = W_0(z) \big(1+ (\phi'(\Re z))^{1/2}\chi_A(z) \big), \qquad z\in \C^+\cup\R,
$$
where $A= \{z=x+iy: 0\le y\le (\phi'(x))^{-1}\}$, and $W_1(z) = W_1(\overline z)$,
$z\in \C^-$. Then $\HH(E) = \FF_{W_1}$.
\end{corollary}

We conclude this section with an example which shows that 
the appearance of the level set is intrinsic in our problem. 
Let $\Theta = E^\sharp/E$ be one-component. Put 
$$
W_2(z) = W_0(z) (1+\|k_z\|_2), \qquad z\in \C^+\cup\R,
$$
and $W_2(z) = W_2(\overline z)$,
$z\in \C^-$. Since $W_2\gtrsim W$, where $W$ is the weight from 
Corollary \ref{one1}, we have $\FF_{W_2}\subset \HH(E)$. 

However, the converse statement need not be true. 
Consider the de Branges space 
$\HH(E)$, defined by the Hermite--Biehler function $E$ 
with zeros 
$\left\{\overline z_n \right\}_{n\in\mathbb{Z}}$, 
such that $z_n= |n|^{\alpha}{\rm sign}\, n +i$, 
$\frac{1}{2}<\alpha<1$, and put $\Theta=E^{\sharp}/E$. 
It is easy to show (see, e.g., \cite{bbh})
that $|\Theta'(x)|\asymp (|x|+1)^{\frac{1}{\alpha} -1}$, $x\in\R$. 
Therefore, $\phi'(x)dx$ is a locally doubling measure and $\Theta$ is one-component.
Moreover, there exists a constant $c>0$ such that for $z=x+iy\in \C^+$,
$$
\|k_z\|_2^2 = \frac{1-\left|\Theta(z)\right|^2}{4\pi y} 
\asymp\begin{cases}x^{\frac{1}{\alpha}-1},\ \left|x\right|\leq c y^{-\frac{\alpha}{1-\alpha}},\\
y^{-1},\ \left|x\right|> c y^{-\frac{\alpha}{1-\alpha}}.
\end{cases} 
$$
We will show that $\HH(E)$ is not contained in $ \FF_{W_2}$. 
Indeed, it is easy to find a function 
$f$ in the model space $K_{\Theta}$ such that 
$\left|f(x+iy)\right|\asymp (|x|+1)^{-1/2} \log(|x|+2)^{-1}$, $x\in\R$, $0\le y\le 2$
(for much more general results of this type see \cite{belov}).
Then $F=fE \in \HH(E)$. However, setting $\gamma=\frac{\alpha}{1-\alpha}$  we get
$$
\int_0^1 \int_{\R} |f(z)|^2\|k_z\|^2_2 \gtrsim
\int^{1}_{0}\frac{1}{y}\int^{+\infty}_{c/y^{\gamma}}\frac{dx}{x\log^{2}x}dx
\asymp\int_{0}^1 \frac{dy}{y|\log y|}=\infty.
$$
Thus, $F\notin \FF_{W_2}$. 
\bigskip


\section{Proof of Theorem \ref{F Wb}}
\label{th1}

The implication (ii)$\Longrightarrow$(i) follows from Theorem \ref{main}. Indeed, 
if $\overline z_n = x_n -iy_n$ is a zero of $E$, then
by \eqref{def phi'} $\phi'(x_n)\ge y_n^{-1}$. 
Thus, the inclusion $\varphi'\in L^{\infty}(\R)$ 
implies that $\inf_n y_n >0$ and 
it follows from \eqref{lev} that $\inf_{x\in \R} d_\vep(x)>0$ 
for any $\vep \in (0,1)$. Therefore, we can choose $0<\vep<\delta<1$
such that for the weight $W$ defined by \eqref{mw} we have 
$W \asymp W_0$.  Thus, $\FF_{W_0} \subset \HH(E)$.

In the implication (i)$\Longrightarrow$(ii) we will use the following lemma.

\begin{lemma}
\label{bab}
If $\FF_{W_0} = \HH(E)$ with equivalence of norms and $Z_E = \{\overline z_n\}$,
then $\inf_{n} \Im z_n>0$.
\end{lemma}

\begin{proof}
Recall that the zeros of $E$ are denoted by $\overline z_n=x_n-iy_n,$ 
with $y_n>0$ for each $n$. Let us consider the following test functions
$$
f_n(z)=E(z)\frac{\sqrt{y_n}}{z-\overline z_n},\qquad z\in\C,\ n\in\N.  
$$
We shall compare ${\FF_W}$ and $\HH(E)$ norms of $f_n$. We have
\begin{align*}
  \left\|f_n\right\|^{2}_{\FF_W}&=\int_{\C^+}\frac{y_n}{\left|z-\overline z_n\right|^2}\frac{dm(z)}{\left(1+\left|\Im z\right|\right)^2}+\int_{\C^-}\frac{y_n}{\left|z-\overline z_n\right|^2}\left|\frac{E(z)}{E(\overline z)}\right|^2\frac{dm(z)}{\left(1+\left|\Im z\right|\right)^2}\\
  &\leq \int_{\C^+}\frac{y_n}{\left|z-\overline z_n\right|^2}\frac{dm(z)}{\left(1+\left|\Im z\right|\right)^2}+\int_{\C^-}\frac{y_n}{\left|z-\overline z_n\right|^2}  \left|\frac{z-\overline z_n}{\overline z-\overline z_n}\right|^2 \frac{dm(z)}{\left(1+\left|\Im z\right|\right)^2}\\
\end{align*}
since for $z\in \C^-$, $z=x-iy$, $y>0$, we have, by \eqref{merom}, 
$$ 
  \left|\frac{E(z)}{E(\overline z)}\right|^2 =e^{-2\tau y}
  \prod_{k\neq n}\left|\frac{z-\overline z_k}
  {\overline z-\overline z_k}\right|^2\left|\frac{z-\overline z_n}
  {\overline z-\overline z_n}\right|^2  \le \left|\frac{z-\overline z_n}
  {\overline z-\overline z_n}\right|^2.
$$
Using the change of variable $w=\overline z$ in the integral over $\C^-$, we get
\begin{align*}
\left\|f_n\right\|^{2}_{\FF_W}&\leq 2\int_{\C^+}\frac{y_n}{\left|z-\overline z_n\right|^2}\frac{dm(z)}{\left(1+\left|\Im z\right|\right)^2}\\
&=2 \int^{+\infty}_{0}\left(\int^{+\infty}_{-\infty}\frac{y_n}{(x-x_n)^2+(y+y_n)^2}dx\right)\frac{dy}{\left(1+y\right)^2}\\
&=2\pi \int^{+\infty}_{0}\frac{y_n}{y+y_n}\frac{dy}{\left(1+y\right)^2}\\
&\lesssim \int^{1}_{0}\frac{y_n}{y+y_n}dy+y_n \int^{+\infty}_{1} \frac{dy}{y\left(1+y\right)^2}\\
&\asymp y_n \log\left(\frac{1+y_n}{y_n}\right)+y_n.
\end{align*}
Since $\left\|f_n\right\|_E\asymp 1$, we conclude that $y_n \gtrsim 1$.
\end{proof}

\begin{proof}[Proof of {\rm (i)$\Longrightarrow$(ii)}]
Suppose that $\FF_W= \HH(E)$ with equivalent norms, and that $\varphi'$ 
is not bounded. Introducing the following test functions
$$ 
     g_x(z)=\frac{K_x(z)}{\overline{E(x)}}=\frac{E(z)}{z-x}
     \left(1-\overline{\Theta(x)}\Theta(z)\right),\qquad  x\in\R,\ z\in\C, 
$$
and assuming that $\phi'(x)>1$, 
we will again compare their ${\FF_W} $ and $\HH(E)$ norms. 
As in Lemma \ref{bab}, we have
\begin{align*}
\left\|g_x\right\|^{2}_{\FF_W}&=\int_{\C^+}\left|\frac{1-\overline{\Theta(x)}
\Theta(z)}{z-x}\right|^{2} \frac{1}{\left(1+\left|\Im z\right|\right)^2}dm(z) 
\\
& \quad +\int_{\C^-}\left|\frac{1-\overline{\Theta(x)}\Theta(z)}{z-x}\right|^{2}  \left|\frac{E(z)}{E(\overline z)}\right|^2  \frac{1}{\left(1+\left|\Im z\right|\right)^2}dm(z)\\
& \quad \quad \asymp \int_{\C^+}\left|\frac{1-\overline{\Theta(x)}\Theta(z)}{z-x}\right|^{2} \frac{1}{\left(1+\left|\Im z\right|\right)^2}dm(z).
\end{align*}
For any positive number $\epsilon<1$, 
let us denote by $S_\vep$ the strip $\{0\le \Im z \le \vep\}$ 
and by $R_{\epsilon}$ the rectangle 
$\left\{z=\xi+i\eta\in\C^+,\ \left|\xi-x\right| 
\le \epsilon,\ 0 \le \eta \le \epsilon \right\}$. 
We split the integral as follows:
$$
\begin{aligned}
  \int_{\C^+} & \left|\frac{1-\overline{\Theta(x)}\Theta(z)}{z-x}\right|^{2} 
  \frac{1}{\left(1+\left|\Im z\right|\right)^2}dm(z)  \\
& =\left(\int_{\left\{\Im z > \epsilon\right\}}  +
\int_{S_\vep \setminus R_{\epsilon}}   +\int_{R_{\epsilon}}\right)\left|\frac{1-\overline{\Theta(x)}\Theta(z)}{z-x}\right|^{2} \frac{1}{\left(1+\left|\Im z\right|\right)^2}dm(z).
\end{aligned}
$$
We have 
\begin{align*}
	\int_{\left\{\Im z\geq\epsilon\right\}}\left|\frac{1-\overline{\Theta(x)}\Theta(z)}{z-x}\right|^{2} \frac{1}{\left(1+\left|\Im z\right|\right)^2}dm(z)
	&\lesssim \int^{\infty}_{\epsilon}\int_{\R}\frac{d\xi}{\left(\xi-x\right)^2+\eta^2}\frac{d\eta}{(1+\eta)^2}\\
	& = \pi \int^{\infty}_{\epsilon}\frac{1}{\eta}\frac{d\eta}{(1+\eta)^2} 
        \asymp 1+\log \frac{1}{\epsilon}.
\end{align*}
Now
\begin{align*}
	\int_{S_{\epsilon}\setminus R_{\epsilon}}\left|
        \frac{1-\overline{\Theta(x)}\Theta(z)}{z-x}\right|^{2} 
        \frac{1}{\left(1+\left|\Im z\right|\right)^2}dm(z)& 
        \lesssim\int^{\epsilon}_{0}\int_{\left\{\left|\xi-x\right|>\epsilon\right\}}
        \frac{d\xi}{\left(\xi-x\right)^2}\frac{d\eta}{(1+\eta)^2}\\
	&\asymp \frac{1}{\vep} \int^{\epsilon}_{0}
        \frac{d\eta}{(1+\eta)^2} \asymp 1.
\end{align*}
To estimate the integral over $R_\vep$, note first that, for $z=x+iy\in\C^+$,
$|\Theta'(z)|\le |\Theta'(x)| = 2\phi'(x)$ (this follows immediately from 
\eqref{ert}). Also, if $\inf_n y_n>0$, it follows from 
\eqref{phase}, that $\phi'(x) \asymp\phi'(t)$, $|x-t|\le\vep$, with the constants 
depending on $\inf_n y_n$ and independent on $\vep\in (0,1)$. We conclude that 
$\max_{u\in R_{\epsilon}}\left|\Theta'(u)\right| \lesssim |\phi'(x)|$. 
	
Since $\overline{\Theta(x)}=1/\Theta(x)$, we get
\begin{align*}
	\int_{R_{\epsilon}}\left|\frac{1-\overline{\Theta(x)}\Theta(z)}{z-x}\right|^{2} \frac{1}{\left(1+\left|\Im z\right|\right)^2}dm(z)
	&\lesssim \int_{R_{\epsilon}}\left|\frac{\Theta(x)-\Theta(z)}{z-x}\right|^{2}dm(z)\\
	&\lesssim \epsilon^2 \max_{u\in R_{\epsilon}}\left|\Theta'(u)\right|^{2} 
        \lesssim \epsilon^2 \left[\varphi'(x)\right]^2.
\end{align*}
Choosing $\epsilon \in(0,1)$  such that $\epsilon^2 
\left[\varphi'(x)\right]^2\asymp \epsilon^{-1}$,  
we have 
$\left\|g_x\right\|^2_{\FF_W}\lesssim \epsilon^{-1} \asymp [\phi'(x)]^{2/3}$
which, altogether with the fact that $\left\|g_x\right\|_{\HH(E)}^2\asymp
\varphi'(x)$, contradicts the equivalence of the norms if $\phi'(x)$ takes 
arbitrarily large values. Thus,  $\varphi' \in L^\infty(\R)$. 
\end{proof}
\bigskip


\section{Weight associated with the spectral data}
\label{spect}                                     

Let $\HH(E)$ be a de Branges space and $\alpha\in [0, \pi)$. Then the function
$E_\alpha(z) = e^{i\alpha}E - e^{-i\alpha}E^\sharp$ has only simple real zeros 
that we denote $t_{n, \alpha}$. It is one of the basic results of the de Branges 
theory (see \cite[Theorem 22]{db}) that the reproducing kernels 
$\{K_{t_{n,\alpha}}\}$ form
an orthogonal basis in $\HH(E)$ for all values of $\alpha$ except at most one.
This exceptional $\alpha$ corresponds to the case when $E_\alpha\in \HH(E)$ 
or, equivalently, $e^{2i\alpha}-\Theta \in H^2$. Also, 
$$
\frac{E_\alpha(z)}{z-t_{\alpha, n}} =
\frac{\pi}{i E(t_{\alpha, n})} K_{t_{\alpha, n}}(z) \quad\text{and}\quad
\Big\| \frac{E_\alpha(z)}{z-t_{\alpha, n}} \Big\|_E^2 
= \frac{\pi}{\phi'(t_{\alpha,n})}.
$$

Since $\HH(E) = \HH(e^{i\alpha} E)$ for any $\alpha$, we can assume
in what follows that $A= (E+E^{\sharp})/2 \notin \HH(E)$ 
(which is equivalent to $1+\Theta\notin H^2$). 
Put $T:=Z_A=\left\{t_n\right\}$ and $\mu_n = [\phi'(t_n)]^{-1}$. 
Then any function
in $\HH(E)$ has the representation
\begin{equation}
   \label{sp}
     F(z)=A(z)\sum_{n}\frac{c_n\mu^{1/2}_{n}}{z-t_n},\qquad
     \left\{c_n\right\}\in \ell^{2},
\end{equation}
where
$$
c_n=\frac{1}{\mu^{1/2}_{n}}\frac{F(t_n)}{A'(t_n)}
$$
and $\|F\|_E^2 = \pi \|\{c_n\}\|_{\ell^2}^2$.
We will call the pair $(T, \mu)$, where $\mu = \sum_n \mu_n \delta_{t_n}$, 
the {\it spectral data} for the de Branges space $\HH(E)$. Note that 
we have $\sum_n (t_n^2+1)^{-1}\mu_n <\infty$. 

It is clear that $\|F\|_E^2 = \pi\int_\R |F/E|^2 d\mu$. We will replace 
the discrete measure $\mu$ by a weight concentrated near the points $t_n$.
Let $r_n>0$ be such that $r_n<{\rm dist}\,(t_n, \{t_j\}_{j\ne n})/2$ and
\begin{equation}
   \label{rn}
      \sum_n \frac{r_n^2}{\mu_n} \sum_{j\ne n} \frac{\mu_j}{|t_n-t_j|^2}<\infty.
\end{equation}
Put $W(z) = W_0(z) +W_T(z)$, where
$$
W_T(z) = \sum_n \frac{1}{\mu_n^{1/2} r_n}\cdot
\frac{|z-t_n|}{|A(z)|}\chi_{D(t_n, r_n)}.
$$

\begin{theorem}
\label{clark}
$\HH(E) = \FF_W$.
\end{theorem}

\begin{proof} {\bf Inclusion} $\HH(E) \subset \FF_W$. 
We know that $\HH(E) \subset \FF_{W_0}$. Let $F\in \HH(E)$ 
have representation \eqref{sp}. Then for each disc $D_n=D(t_n, r_n)$ we have
\begin{align*}
\int_{D_n}\left|FW_T\right|^2 dm&=\int_{D_n}\bigg| 
\sum_{j\in\N}\frac{c_j\mu^{1/2}_{j}}{z-t_j} \bigg|^2 
\frac{\left|z-t_n\right|^2}{\mu_n r^{2}_{n}}dm(z) \\
&\le 2 \int_{D_n}\left(\frac{\left|c_n\right|^2}
{r^{2}_{n}}+\frac{\left|z-t_n\right|^2}{\mu_n r^{2}_{n}}
\bigg|\sum_{j\neq n}\frac{c_j \mu^{1/2}_{j}}{z-t_j}\bigg|^2\right)dm(z) \\
&\le 2\pi |c_n|^2 + 2
\int_{D_n} \frac{\left|z-t_n\right|^2}{\mu_n r^{2}_{n}}
\bigg(\sum_{j\neq n}\frac{\mu_{j}}{\left|z-t_j\right|^2}\bigg) 
\left\|c\right\|^{2}_{l^2} dm(z)
\end{align*}
Since for $z\in D_n$ and $j\neq n$, we have $\left|z-t_j\right|
>\frac{\left|t_n-t_j\right|}{2}$, and 
$$  
\int_{D_n}\left|z-t_n\right|^2dm(z)\asymp r^{4}_{n},
$$
we get
$$ 
\int_{D_n}\left|FW_T\right|^2 dm 
\lesssim\left|c_n\right|^2+\left\|c\right\|^{2}_{l^2} 
\bigg(\sum_{j\neq n}\frac{\mu_{j}}{\left|t_n-t_j\right|^2}\bigg)
\frac{ r^{2}_{n}}{\mu_n}.
$$
Summing the integrals over $D_n$  and using condition \eqref{rn} 
we conclude that \medskip
$F\in \FF_{W_T}$.

{\bf Inclusion} $\FF_W \subset \HH(E)$.
Let $F\in \FF_W$. Using analyticity of the function 
$\frac{F(z)(z-t_n)}{A(z)}$ in $D_n$, we have, for each $n$, 
\begin{align*}
\frac{1}{\mu_n}\left|\frac{F(z)(z-t_n)}{A(z)}\right|^2_{z=t_n}&=\frac{1}{\mu_n}\left|\frac{F(t_n)}{A'(t_n)}\right|^2\\
&\leq\frac{1}{\pi r^{2}_{n}\mu_n}\int_{D_n}\left|\frac{F(w)}{A(w)}\right|^{2}\left|w-t_n\right|^2 dm(w)\\
&=\frac{1}{\pi}\int_{D_n}\left|F(w)W_T(w)\right|^2dm(w).
\end{align*}
Then, for 
$$
c_n:=\frac{1}{\mu^{1/2}_{n}}\frac{F(t_n)}{A'(t_n)},
$$
we have $\{c_n\}\in\ell^2$. To show that $F\in \HH(E)$ we need to prove that  
$$
\frac{F(z)}{A(z)}=\sum_{n} \frac{c_n\mu^{1/2}_{n}}{z-t_n}, \qquad z\in \C.
$$
To do this, it is enough to show that the entire function
$$
H(z)=\frac{F(z)}{A(z)}-\sum_{n} \frac{c_n\mu^{1/2}_{n}}{z-t_n}, \qquad z\in \C,
$$
is identically zero ($H$ is entire because $A$ has simple zeros and so
the residues of $H$ at $t_n$ are zero). 
 
We know from the inclusion $F\in \FF_{W_0}$ that
$$
\bigg|\frac{F(z)}{E(z)}\bigg| \lesssim 1 +\frac{1}{|\Im z|}, \qquad z\in \C^+
$$
(see Section \ref{th2}). Recall that 
$A=E(1+\Theta)/2$. It follows from formulas \eqref{ert}--\eqref{ert0} that 
$|1+\Theta(z)|\ge 1-|\Theta(z)| \gtrsim \frac{\Im z}{|z|^2+1}$, whence
$$
\bigg| \frac{F(z)}{A(z)} \bigg| 
\lesssim (|z|^2+1)\bigg(\frac{1}{|\Im z|^2} +\frac{1}{|\Im z|}\bigg), \qquad z\in \C^+.
$$
The same estimate holds for $z\in\C^-$ since we can apply the above reasoning to 
$F^\sharp$. Finally, 
$$
\bigg|\sum_{n} \frac{c_n\mu^{1/2}_{n}}{z-t_n}\bigg|^2 \le \|\{c_n\}\|^2_{\ell^2}
\sum_{n} \frac{\mu_n}{|z-t_n|^2}\lesssim
\sum_n \frac{\mu_n}{t_n^2}\bigg|\frac{t_n}{z-t_n}\bigg|^2 \lesssim 
\frac{|z|^2}{|\Im z|^2}.
$$
Thus, $|H(z)| \lesssim (|z|^2+1)|(|\Im z|^{-2}+|\Im z|^{-1})$. Since
$$
\log|H(z)| \le \frac{1}{2\pi} \int_{|\zeta-z|=1} \log|H(\zeta)|\,|d\zeta|,
$$
we conclude that $|H(z)| \le C(|z|^2+1)$, $|\Im z| \le 1$. 
Thus, $H$ is a polynomial of degree at most $1$.  

We have
$$
F(z) = A(z)H(z) + A(z)\sum_{n}\frac{c_n\mu^{1/2}_{n}}{z-t_n}.
$$
The latter sum is in $\HH(E)$ and we already know that $\HH(E) \subset \FF_W$.
Hence, $AH\in \FF_W$ and in particular, $AH\in \FF_{W_0}$.
Since $A=(1+\Theta)E/2$, we conclude that, for any $y>0$,
$$
\int_{\R} |1+\Theta(x+iy)|^2 |H(x+iy)|^2 dx  <\infty.
$$
If $H$ is a nonzero polynomial we conclude that $1+\Theta(\cdot+iy) \in L^2(\R)$
for any $y>0$ and so $1+\Theta(\cdot +iy) \in H^2$. From this it is easy to deduce
that $1+\Theta\in H^2$, a contradiction to the fact that $A\notin \HH(E)$ 
by the choice of the spectral data. To see that $1+\Theta\in H^2$ one can use, 
e.g., the results of \cite{bar2}. By \cite[Theorem 2]{bar2}, for a nonconstant 
bounded analytic function $f$ in $\C^+$ with $\|f\|_\infty\le 1$, 
the inclusion $1-f\in H^2$ is equivalent to the asymptotics
$$
f(iy) = 1 - \frac{q}{y} +o\Big(\frac{1}{y}\Big), \qquad y\to+\infty,
$$
with some $q>0$. Thus, 
$$
\Theta(iy+iy_0) = -1 + \frac{q}{y} +o\Big(\frac{1}{y}\Big), \qquad y\to+\infty.
$$
Then, by the same theorem, $1+\Theta \in H^2$ 
(since it has the same asymptotics as $y\to+\infty$). This contradiction 
shows that $H\equiv 0$.
\end{proof}
\bigskip


\section{Weights depending on the imaginary part}
\label{th3}                                     

In this section we prove Theorem \ref{depy}. We need to prove only implication 
(i)$\Longrightarrow$(ii). Therefore, in what follows we assume that
$W(z) = \Phi(|y|)$, $y=\Im z$, where $\Phi: [0, \infty)\to (0, \infty)$
and, in view of \eqref{bab1}, 
$\inf_I \Phi>0$ for any bounded interval $I\subset [0, \infty)$.
We write $\FF_\Phi$ in place of $\FF_W$.

\begin{lemma}
\label{l2}
If $\HH(E) = \FF_\Phi$,  then there exists $M>0$ such that
\begin{equation}
   \label{rep}
M^{-1}\|K_{iy}\|_{E} \le \|K_{x+iy}\|_{E} \le M\|K_{iy}\|_{E}, 
\qquad x, y\in \R.
\end{equation}
In particular, for any $F\in \HH(E)$ and $y_1, y_2\in \R$, 
$F$ is bounded in $\{z: y_1\le \Im z\le y_2\}$.
\end{lemma}

The proof of the lemma follows immediately from the fact that 
the translation operators $F \mapsto  F(\cdot + t)$, $t\in\mathbb{R}$, are 
unitary on $\FF_\Phi$ and from the formula  
$\|K_z\|_E = \sup_{F\in \HH(E), \|F\|_E\le 1} |F(z)|$.

\begin{lemma}
\label{expo}
If $\HH(E) = \FF_\Phi$,  then $E$ is of finite exponential type.
\end{lemma}

\begin{proof} Recall that $\|K_z\|_E^2 = 
|E(z)|^2 \frac{1-|\Theta(z)|^2}{4\pi \Im z}$, $z\in \C^+$. 
Then it follows from \eqref{rep} that 
\begin{equation}
   \label{trot}
|E(iy)|^2 (1-|\Theta(iy)|^2) \asymp
|E(x+iy)|^2 (1-|\Theta(x+iy)|^2), \qquad x, y\in \R,
\end{equation}
with the constants independent on $x, y$. Note that, by the formulas 
\eqref{ert}--\eqref{ert0}, 
$1-|\Theta(z)|^2 \gtrsim |z|^{-2}\Im z \gtrsim |z|^{-2}$, $\Im z\ge 1$, 
and so $|E(x+iy)| \lesssim |E(i|y|)|\cdot|z|$, $|y| \ge 1$. 

Also, by Lemma \ref{l2}, $|E(z)|\lesssim 1+|z|$, $-1\le \Im z\le 1$,
since $E/(z-z_0) \in \HH(E)$ for any zero $z_0$ of $E$. Thus, 
if $\liminf_{y\to +\infty} |E(iy)| <\infty$, then $E$ is a polynomial of degree
at most 1 by the Phragm\'en--Lindel\"of principle applied to the strips. 

Now assume that $\lim_{y\to +\infty}|E(iy)|=+\infty$. We also have, by 
\eqref{trot}, $|E(x+iy)| \geq y^{-1}|E(iy)|$, $y\ge 1$.
Hence, there exists
$C>0$ such that $C|zE(z)|\ge 2$, $\Im z\ge 1$. Then $h(z) = \log C|zE(z)|$
is a positive harmonic function in $\{\Im z\ge 1\}$, and we can write its Riesz
representation 
$$
h(z) = p(y-1) + \frac{y-1}{\pi}\int_\R  \frac{d\mu(t)}{(t-x)^2+(y-1)^2},
\qquad z=x+iy, \ y\ge 1,
$$
with some measure $\mu$ such that $\int_\R(1+t^2)^{-1} d\mu(t) <\infty$.
By obvious estimates of the Poisson integral, 
we get, for any $\vep\in (0, \pi)$,
\begin{align*}
|h(z)| & \lesssim |z|^2, \qquad \Im z\ge 2 \\
|h(z)| & \lesssim |z|, \qquad \vep\le \arg z\le \pi-\vep, \ \Im z\ge 2.
\end{align*}
Since $E$ is in the Hermite--Biehler class, analogous estimates hold 
in the lower half-plane. Since, moreover, $|E(z)| \lesssim |z|$ in 
$\{-2\le \Im z \le 2\}$, we conclude by the 
Phragm\'en--Lindel\"of principle applied to the angles $\{|\arg z|\le \vep\}$
and $\{|\arg z - \pi|\le \vep\}$ that $E$ is of finite exponential type. 
\end{proof}
\medskip

\begin{proof}[Proof of Theorem \ref{depy}] Let $\HH(E) = \FF_\Phi$. First we will 
show that $F(\cdot+ iy_0) \in L^2(\R)$ for any $F\in \HH(E)$ and $y_0\in\R$. Indeed, 
$$
|F(x+iy_0)|^2 \le \int_{D(x+iy_0, 1)} |F(\zeta)|^2 dm(\zeta),
$$
whence
\begin{gather*}
\int_\R  |F(x+iy_0)|^2 dx \lesssim 
\int_{y_0-1\le\Im \zeta\le y_0+1} |F(\zeta)|^2 dm(\zeta) \qquad \qquad\\
 \le \int_{y_0-1\le\Im \zeta\le y_0+1} |F(\zeta)\Phi(\zeta)|^2 
dm(\zeta) \cdot \sup\limits_{y_0-1\le\Im \zeta\le y_0+1} [\Phi(y)]^{-2} \le 
M^{-2} \|F\|^2_{\FF_\Phi},
\end{gather*}
where $M = \inf_{y_0-1\le y \le y_0+1} \Phi(y)>0$.

By Lemma \ref{expo}, $E$ is of finite exponential type and so 
the number $a$ defined by
$$
a = \limsup\limits_{y\to+\infty} \frac{\log|E(iy)|}{y}
$$
is finite. Also, for any zero $z_0$ of $E$, the function $G(z) = E(z)/(z-z_0)$
is in $\HH(E)$ and so in $L^2(\R)$. Since $G$ is of finite exponential type,
we conclude that $G$ belongs to some Paley--Wiener space. A function in the 
Paley--Wiener space has maximal growth along the imaginary axis and so $G\in PW_a$. 

In particular, $e^{iaz} G \in H^2$.
The function $G$ does not vanish in $\C^+$ and also 
$\limsup_{y\to+\infty} y^{-1} \log |e^{-ay}G(iy)| = 0$. We conclude that $e^{iaz} G$ 
is an outer function in $\C^+$
(since its canonical Smirnov--Nevanlinna factorization contains no factors 
of the form $e^{i\tau z}$, $\tau >0$), and it follows from standard 
estimates  of the Poisson integral that 
\begin{equation}
   \label{esio}
   a = \lim\limits_{y\to+\infty} \frac{\log|E(x+iy)|}{y}
\end{equation}
and the limit is uniform for $x\in [-1,1]$.

Now let $F\in\HH(E)$. Then $F\in L^2(\R)$ and $F/E$ is in $H^2$. Hence, 
$$
\limsup\limits_{y\to+\infty} \frac{\log|e^{-ay}F(iy)|}{y} = 
\limsup\limits_{y\to+\infty} \frac{\log|e^{-ay}E(iy)| + \log|E^{-1}(iy)F(iy)|}{y} \le 0.
$$                                                                  
Thus, $F\in PW_a$ and we conclude that $\HH(E) \subset PW_a$.

It remains to prove the converse inclusion $PW_a\subset \HH(E)$.  Put
$$
b= \sup\bigg\{b': \int_0^\infty e^{2b'y}\Phi^2(y)dy<\infty  \bigg\}.
$$
Since
$$
\int_{-1}^1 \int_0^\infty \frac{|E(x+iy)|^2}{|x+iy-z_0|^2}\Phi^2(y)dy dx <\infty
$$
and $|E(x+iy)|^2 \gtrsim e^{2a'y}$, $y\ge 0$,
for any $a'<a$ by \eqref{esio}, 
we conclude that $a\le b$. Now let $b'<b$. Then $PW_{b'} \subset \FF_\Phi$. Indeed, 
$\int_\R |f(x+iy)|^2 dx \le\|f\|^2_2 e^{2b'|y|}$ for any $f\in PW_{b'}$, whence
$$
\int_\C |f(x+iy)|^2 \Phi^2(|y|)dxdy \le 2\|f\|^2_2 \int_{0}^\infty e^{2b'y}
\Phi^2(y) dy<\infty.
$$
Thus, $PW_{b'} \subset \FF_\Phi \subset PW_a$ for any $b'<b$ and so $b\le a$.  

We have seen that $PW_{b'} \subset \FF_\Phi \subset PW_a$ for any $b'<a$. 
To complete the proof note that for any $g\in PW_{b'}$ 
we have $g/E \in L^2(\R)$. Now let $f\in PW_a$. Then 
$f = e^{iaz/2} f_1 + e^{-iaz/2} f_2$
where $f_1, f_2 \in PW_{a/2}$ and so $f/E \in L^2(\R)$. 
By the discussion in the beginning of Section \ref{th2}, it remains to show that $f/E$,
$f^\sharp/E$ are in the Smirnov class $\mathcal{N}_+$. Since both $f$ and $E/(z-z_0)$
are in $PW_a$, the function $f/E$ is of bounded type in $\C^+$. Finally, 
$$
\limsup\limits_{y\to+\infty} \frac{1}{y}\log\bigg|\frac{f(iy)}{E(iy)}\bigg| = 
\limsup\limits_{y\to+\infty} \bigg(\frac{1}{y}\log |e^{-ay} f(iy)| +
\frac{1}{y}\log \frac{e^{-ay}}{|E(iy)|}\bigg) \le 0. 
$$
Thus, $f/E$ (and similarly, $f^\sharp/E$) are in $H^2$, whence $f\in \HH(E)$. 
\end{proof}

\subsection{Final remarks}
Consider the following three statements: 
\smallskip
\begin{enumerate}
	\item [(i)] $\HH(E) = \FF_{\Phi}$ ($\Longleftrightarrow \HH(E) = PW_a$);
\smallskip
	\item [(ii)] $\sup_{t\in\R}\|U_t:\HH(E)\to\HH(E)\|<\infty$;
\smallskip
	\item [(iii)] $\|K_x\|_E \asymp \|K_0\|_E$, $x\in\R$,
i.e., $\phi'(x)|E(x)|^2 \asymp 1$.
\end{enumerate}
\smallskip
We have seen that (i)$\Longrightarrow$(ii)$\Longrightarrow$(iii). One can ask whether
converse to any of these implications is true. We do not know whether
(iii)$\Longrightarrow$(ii) or (ii)$\Longrightarrow$(i), but we can show that
(iii) does not imply (i). Namely, we will construct a de Branges space
that satisfies (iii),  but does not coincide as a set with the Paley--Wiener space.

\begin{example}
{\rm Given $\delta>0$, consider the sequence $(z_n)_{n\in\Z}$ defined by
$$  
z_n=\begin{cases}
n-\delta+in^{-4\delta},\ n>0,\\
n+\delta+i\left|n\right|^{-4\delta},\ n<0,\\
i,\ n=0,
\end{cases}
$$
and let $E(z)=\lim_{R\to\infty} \prod_{|z_n|\le R}(1-z/\overline z_n)$
be the corresponding Hermite--Biehler function.  
This product converges uniformly on compact sets and $E$
is a function of Cartwright class. This example 
was given in \cite{ls} where it was shown that $\HH(E) = PW_\pi$ 
if and only if $\delta<1/4$. However, we will show that for $1/4 \le\delta\le 1/2$
the function $E$ satisfies (iii).

In what follows let $\delta\in (0,1)$. We will estimate 
$E(x)$, $x\in \R$, by comparing it to the function $\sin \pi(x+i)$ whose modulus is 
comparable to a constant. These are well-known calculations which we include for the 
sake of completeness. Let $|x-k|\le 1/2$ and $k>0$. Then it is easy to show that
$$
\left|\frac{E(x)}{\sin \pi(x+i)}\right|
\asymp \prod_{n\in\Z}\left|\frac{1-\frac{x}{\overline z_n}}{1-\frac{x}{n-i}}\right|^{2}
\asymp A^{2}\left|\frac{1-\frac{x}{\overline z_k}}{1- \frac{x}{k-i}}\right|^{2},
$$
where 
$$
A=\prod_{n\geq 1, n\neq k}\left| \frac{\left(1-\frac{k}{n-\delta}\right)\left(1-\frac{k}{n+\delta}\right)}{\left(1-\frac{k}{n}\right)\left(1+\frac{k}{n}\right)} \right|.
$$
Note that 
$$
A= \prod_{n\geq 1, n\neq k}\left|1+\frac{k^2\left(-2n\delta+\delta^2\right)}
{\left(n-\delta\right)^{2}\left(n^2-k^2\right)} \right|
\asymp \prod_{n\geq 1, n\neq k}\left|  1 + \frac{2\delta nk^2}
{\left(n-\delta\right)^{2}\left(k^2-n^2\right)} \right|.
$$
We have also 
$$
\sum_{n\geq 1, n\neq k}\frac{n k^2}{\left(n-\delta\right)^{2}
\left(k^2-n^2\right)}\\
= \sum_{n\geq 1, n\neq k}\frac{k^2}{n\left(k^2-n^2\right)}+O(1) 
= \sum_{1\leq n\leq 2k-1, n\neq k} \frac{k^2}
{n\left(n^2-k^2\right)}+O(1)
$$
(where $O(1)$ denotes a quantity uniformly bounded in $k$) and
\begin{align*}
\sum_{1\leq n\leq 2k-1} \frac{k^2}{n\left(k^2-n^2\right)}
& =\sum_{1\leq n\leq 2k-1}\left[\left(\frac{1}{n}-\frac{1}{k+n}\right)+\frac{1}{2}
\left(\frac{1}{k-n}+\frac{1}{k+n}\right)\right]\\
&=\sum_{1\leq n\leq 2k-1}\left[\frac{1}{n}- \frac{1}{2(k+n)} \right]
= \log k+O(1).
\end{align*}
Thus 
\begin{align*}
\log A  & = 
\sum_{n\geq 1, n\neq k} \log\left|1 + \frac{2\delta nk^2}
{\left(n-\delta\right)^{2}\left(k^2-n^2\right)} \right| \\
& = 2\delta \sum_{n\geq 1, n\neq k}\frac{k^2}{n\left(k^2-n^2\right)}+O(1) = 2\delta\log k+O(1)
\end{align*}
and so
$$  
\left|E(x)\right|^2\asymp \left|k\right|^{4\delta}\left|x-z_k\right|^2,\qquad x\in\R, \ 
|x-k|\le 1/2.
$$

Now we estimate the derivative of the phase function: 
$$
\varphi'(x)=\sum_{n\geq 1, n\ne k} \frac{y_n}{\left(x-n+\delta\right)^2+y_n^2}
+\sum_{n< 0}  \frac{y_n}{\left(x-n-\delta\right)^2+y_n^2} + 
\frac{1}{x^2+1}
+\frac{k^{-4\delta}}{\left(x-k+\delta\right)^2+k^{-8\delta}}.
$$
By simple estimates of sums,
\begin{align*}
\sum_{|n|\ge 3k/2}\frac{n^{-4\delta}}{(x-n)^2}  \lesssim k^{-4\delta-1},
& \qquad
\sum_{k/2\le n \le 3k/2, n\ne k} \frac{n^{-4\delta}}{(x-n)^2}  \lesssim k^{-4\delta}, \\
\sum_{-3k/2\le n \le k/2, n\ne 0} \frac{|n|^{-4\delta}}{(x-n)^2} & \lesssim\frac{1}{k^2}
\sum_{-3k/2\le n \le k/2, n\ne 0} |n|^{-4\delta}.
\end{align*}
If $0<4\delta<1$, then the last expression is $\asymp k^{-4\delta-1}$, while 
for $1< 4\delta \le 2$ it is comparable with $k^{-2}\lesssim k^{-4\delta}$.
Finally, note that
$$
\frac{k^{-4\delta}}{\left(x-k+\delta\right)^2+k^{-8\delta}} \gtrsim k^{-4\delta},
\qquad |x-k|\le 1/2.
$$
We conclude that $\phi'(x) \asymp k^{-4\delta} |x-z_k|^{-2}$, $|x-k|\le 1/2$, 
when $0<\delta\le 1/2$. Thus, $\phi'(x)|E(x)|^2\asymp 1$ for $0<\delta\le 1/2$ 
(and it is easy to see from the above estimates that it is no longer true for
$\delta>1/2$). However, for $1/4 \le\delta\le 1/2$, $\HH(E) \ne PW_\pi$.
\medskip
\\
{\bf Problem.} It would be interesting  to describe the 
de Branges spaces $\HH(E)$ satisfying (iii) or, especially, (ii). }
\end{example}


\bibliographystyle{plain}

\end{document}